\documentclass[12pt]{amsproc}
\usepackage[foot]{amsaddr}
\usepackage{amsmath,amssymb,amsthm}
\usepackage{color,graphics,srcltx,float,anysize}
\usepackage{stmaryrd}
\usepackage{enumerate}
\usepackage{booktabs}
\usepackage[table,x11names]{xcolor}
\usepackage{longtable}
\usepackage{makecell}
\usepackage{hyperref}
\usepackage{soul}

\definecolor{darkblue}{rgb}{0,0,0.8}

\newtheorem{theorem}{Theorem}[section]
\newtheorem{lemma}[theorem]{Lemma}
\newtheorem{corollary}[theorem]{Corollary}
\newtheorem{proposition}[theorem]{Proposition}

\theoremstyle{definition}

\newtheorem{problem}[theorem]{Problem}

\DeclareMathOperator{\PSL}{{\mathrm{PSL}}}

\DeclareMathOperator{\Aut}{{\mathrm{Aut}}}
\DeclareMathOperator{\Sym}{{\mathrm{Sym}}}
\DeclareMathOperator{\Inn}{{\mathrm{Inn}}}

\newcommand\cc[1]{\textnormal{\texttt{#1}}}

\renewcommand{\geq}{\geqslant}
\renewcommand{\le}{\leqslant}
\renewcommand{\ge}{\geqslant}

\usepackage{color}

\newcommand\sd{\mkern3mu{:}\mkern3mu} 

\allowdisplaybreaks

\title{Synchronising primitive groups of diagonal type exist}
\author[Bamberg]{John Bamberg}
\author[Giudici]{Michael Giudici}
\author[Lansdown]{Jesse Lansdown}
\author[Royle]{Gordon F. Royle}

\address{Centre for the Mathematics of Symmetry and Computation, Department of Mathematics and Statistics, The University of Western Australia, Crawley, WA 6009, Australia.}
\email{firstname.lastname@uwa.edu.au}

\dedicatory{Dedicated to the memory of Peter M. Neumann.}

\begin{document}

\maketitle

\begin{abstract}
Every synchronising permutation group is primitive and of one of three types:
\emph{affine}, \emph{almost simple}, or \emph{diagonal}.
We exhibit the first known example of a synchronising diagonal type group. More precisely, we show that $\PSL(2,q)\times \PSL(2,q)$
acting in its diagonal action on $\PSL(2,q)$ is separating, and hence synchronising, for $q=13$ and $q=17$.
Furthermore, we show that such groups are non-spreading for all prime powers $q$.
\end{abstract}

\section{Introduction}

A finite permutation group $G$ acting on a finite set $\Omega$ is \emph{synchronising} if
the automaton whose transitions are generators of $G$ together with an arbitrary non-permutation is synchronising. This means that there is a word in the automaton's alphabet (called a \emph{reset} word) such that, after reading this word, the automaton is in a known fixed state, regardless of its starting state.
Every synchronising permutation group is
primitive \cite[Theorem 3.2]{survey}, and moreover, it has been shown that they are 
of \emph{affine}, \emph{almost simple}, or \emph{diagonal} type \cite[Proposition 3.3]{survey}.
There are many examples known of synchronising groups of affine or almost simple type, but the existence of synchronising groups of diagonal type has been open ever since the birth of this subject (see \cite{Bray}). Bray et. al. \cite{Bray} showed that the structure of a putative synchronising group of diagonal type is constrained, in that its socle contains just two factors. 
In addition, they showed that if $G$ is primitive of diagonal type, then 
$G$ is synchronising if and only if it is \emph{separating}.   
So for primitive permutation groups of diagonal type, we have the following hierarchy:
\[
\text{spreading } \implies \text{ separating } \iff \text{ synchronising}.
\]
See Section \ref{sect:background} for relevant background theory about the synchronisation hierarchy of primitive permutation groups.

In this paper, we investigate the group $\PSL(2,q)\times \PSL(2,q)$ in its diagonal action on $\PSL(2,q)$ and find the first known examples of synchronising groups of diagonal type.

\begin{theorem}\label{main1}
The group $\PSL(2,q)\times \PSL(2,q)$ in its diagonal action on $\PSL(2,q)$ is separating, and hence synchronising, for $q=13$ and $q=17$.
\end{theorem}

We were led to this family of groups because they form a natural class of diagonal type groups whose socle has only two factors. Neumann \cite[Example 3.3]{Neumann}, see also \cite[\S6.4]{survey}, observed that if a nonabelian simple group $T$  admits an \emph{exact factorisation} into two nontrivial proper subgroups $A$ and $B$ (that is, $T=AB$ with $A\cap B=\{1\}$), then $T\times T$ in its diagonal action on $T$ is non-synchronising (and therefore non-separating and non-spreading). It\^{o} \cite{ito} has shown that $\PSL(2,q)$ has an exact factorisation unless $q\equiv 1\pmod 4$ and $q \notin \{5,29\}$. As a result, the first few values of $q$ for which $\PSL(2,q) \times \PSL(2,q)$ might be synchronising are $q=9$, $q=13$, $q=17$.

The group $\PSL(2,9)$ does not have an exact factorisation into subgroups, but we use a slightly more general concept to show that $\PSL(2,9) \times \PSL(2,9)$ is non-synchronising, thereby resolving its position in the synchronisation hierarchy (Proposition \ref{A6}). Then for $q=13$ and $q=17$, we use a combination of theory and computation to show that $\PSL(2,q) \times \PSL(2,q)$ acting in its diagonal action on $\PSL(2,q)$ is synchronising. Finally, we show that none of the groups in this family are spreading.
\begin{theorem}\label{main2}
The group $\PSL(2,q)\times \PSL(2,q)$ acting in its diagonal action on $\PSL(2,q)$, for $q$ a prime power, is non-spreading.
\end{theorem}

This leaves the following open problem to completely resolve the synchronisation hierarchy for groups of this type.

\begin{problem}
For which prime powers $q$ such that $q \equiv 1 \pmod{4}$, $q \geq 25$ and $q \ne 29$ is the group $\PSL(2,q) \times \PSL(2,q)$ acting in its diagonal action on $\PSL(2,q)$ synchronising?
\end{problem}

The smallest $q$ for which this is unresolved is $q=25$. Any argument for the general case will not be completely straightforward because the group is synchronising for $q=13$ and $q=17$, but not synchronising for $q=29$.

To demonstrate that a specific group $G$ is separating or synchronising involves (at least in principle) determining the clique number, coclique number and/or chromatic number for every single $G$-invariant graph. These parameters are notoriously hard to compute for even a single graph, and there can be a large number of $G$-invariant graphs that must be considered. For this reason, we strive to eliminate as many graphs as possible by  theoretical means when proving Theorem \ref{main1}. The technique that we introduce results in a significant reduction in the number of graphs that must be considered by computer.

\section{Background}\label{sect:background}

\subsection{Synchronising, separating, spreading}

A transformation semigroup on a set $\Omega$ 
is called \emph{synchronising} if it contains a constant map (that is, a map $f$ and $\beta\in \Omega$ such that $\alpha^f=\beta$ for all $\alpha\in\Omega$).
A finite permutation group $G$ acting on a set $\Omega$ is \emph{synchronising} if
for any non-bijective transformation $f$, the transformation semigroup $\langle G,f\rangle$ is synchronising. Alternatively, each of the following conditions is equivalent to $G$ being non-synchronising:
\begin{enumerate}
\item There is a nontrivial $G$-invariant graph $\Gamma$
with clique number $\omega(\Gamma)$ equal to its chromatic number (see \cite[Corollary 4.5]{survey}). 
\item There is a nontrivial partition $\mathcal{P}$ of $\Omega$ with transversal $B$ such that $B$ is a transversal of $\mathcal{P}^g$ for all $g\in G$ (see \cite[Theorem 3.8]{survey}).
\end{enumerate}

We need some notation for both sets and multisets that will be used throughout the paper. Suppose that $A$ is a set or multiset whose elements belong to $\Omega$. Then the \emph{characteristic vector} $\chi_A$ of $A$ is the vector indexed by $\Omega$ where $(\chi_A)_\alpha$ is the multiplicity of $\alpha$ in $A$. The cardinality of $A$ is denoted either $|A|$ or $|\chi_A|$.  

A transitive permutation group $G$ acting on a set $\Omega$ is \emph{non-spreading} if there
exists a nontrivial (that is, nonconstant and more than one element with nonzero multiplicity) multiset $A$, a nontrivial set $B$, and a positive integer $\lambda$, such that 
\begin{enumerate}
\item $|A|$ divides $|\Omega|$,
\item $|\chi_A \circ \chi_{B^g}|=\lambda$ for all $g\in G$, where $\circ$ is the Schur product of two vectors.
\end{enumerate}
If in addition, $\lambda = 1$, then $A$ must be a set such that $|A||B|=|\Omega|$, in which case $G$ is said to be \emph{non-separating}
\cite[\S5.5]{survey}.
We can also use invariant graphs
to characterise non-separating group actions.
A transitive permutation group $G$ acting on a set $\Omega$ is non-separating if and only if
there is a nontrivial $G$-invariant graph $\Gamma$ having $\alpha(\Gamma)\cdot \omega(\Gamma) = |V\Gamma|$ (see \cite[Theorem 4.5]{survey}),
where $\alpha(\Gamma)$ and $\omega(\Gamma)$ are the size of the largest coclique and clique of $\Gamma$, respectively. 

Each of the properties synchronising, separating, and spreading are propagated to overgroups. That is, if $G\le H\le \Sym(\Omega)$
and $G$ is synchronising (resp. separating, resp. spreading) then $H$ is also synchronising (resp. separating, resp. spreading).

\subsection{Association schemes}

Let $\Omega$ be a set, and let $A_0, A_1, \ldots, A_d$ be symmetric  $\{0,1\}$-matrices with rows and columns indexed by $\Omega$. Then $\mathcal{A}=(\Omega, \{A_0, A_1, \ldots, A_d\})$ is an \emph{association scheme} if the following conditions hold:
\begin{enumerate}
    \item $A_0$ is the identity matrix,
    \item $\sum_{i=0}^d A_i$ is the matrix with every entry equal to $1$,
    \item There exist constants $p_{ij}^k$ depending only on $i$, $j$, and $k$, such that 
    \[
    A_i A_j = \sum_{k=0}^d p_{ij}^k A_k.
    \]
\end{enumerate}
The matrices $A_0$, $A_1$, $\ldots$, $A_d$ are the \emph{adjacency matrices} of $\mathcal{A}$. Indeed each $A_i$ is the adjacency matrix
of an undirected graph. We shall refer to these as the \emph{graphs in the scheme~$\mathcal{A}$}. A graph is the \emph{union of graphs in the scheme~$\mathcal{A}$} if its adjacency matrix is the sum of adjacency matrices of $\mathcal{A}$.
It is well known that $\mathbb{R}^\Omega$ decomposes into $d+1$ simultaneous eigenspaces for the adjacency matrices of $\mathcal{A}$. Moreover there are projection matrices $E_0, E_1, \ldots, E_d$ onto each of these eigenspaces, such that 
\[
E_i = \sum_{j=0}^d Q_{ji} A_j,
\]
where $Q$ is called the \emph{matrix of dual eigenvalues}.
If $C$ is a subset of $\Omega$, then its \emph{inner distribution} is the vector $a = (a_0,a_1,\ldots,a_d)$ defined by
\[
a_i= \frac{1}{|C|}\chi_C A_i \chi_C^\top.
\]
If $Q$ is the matrix of dual eigenvalues of $\mathcal{A}$, then
\[
(a Q)_j = \frac{|\Omega|}{|C|} \chi_C E_j \chi_C^\top
\]
for all $j\ge 0$. The vector $aQ$ is sometimes known as the \emph{MacWilliams transform} of $C$.

The \emph{dual degree set} of $C$ is the set of nonzero indices $j$
for which the $j$-th coordinate of its MacWilliams transform is nonzero.
Two subsets of $\Omega$ are \emph{design-orthogonal} if
their dual degree sets are disjoint.
We refer the reader to the classic text by Bannai and Ito \cite{BannaiIto1984} for more information on this subject.

\begin{theorem}[\cite{delsarte}, Theorem 3.9 (and discussion thereafter); see also \cite{gm}] \label{thm:CliquesInSchemes}
Let $\mathcal{A}$ be an association scheme on $\Omega$ and let $\Gamma$ be a union of graphs of the scheme.
If $C$ is a clique and $S$ is a coclique in $\Gamma$,
then $|C|\cdot|S|\le |\Omega|$. Equality holds if and only if $C$ and $S$ are design-orthogonal.
\end{theorem}

We will need the following simple observation,

\begin{corollary}[cf., {\cite[3.3 Corollary]{roos}}]\label{intersect1}
Let $\mathcal{A}$ be an association scheme on $\Omega$ and let $\Gamma$
 be a union of graphs in the scheme. If $C$ is a clique and $S$ is a coclique in $\Gamma$,
such that $|C|\cdot|S|= |\Omega|$, then $|C\cap S|=1$.
\end{corollary}

\subsection{Diagonal actions}\label{sec:diagonal}

Let $T$ be a nonabelian simple group. Then
$G=T\times T$ acts on $\Omega=T$ in the \emph{diagonal action} as follows:
\[
t^{(x,y)}:=x^{-1}ty.
\]
This action is always primitive. Note that the direct factors $1\times T$ and $T\times 1$ are regular normal subgroups of $G$.  
Moreover, $G_{1_T}=\{(t,t)\mid t\in T\}$ induces the inner automorphism group $\Inn(T)$ on $\Omega$, and so its orbits on $\Omega$ are the conjugacy classes of $T$.

An \emph{orbital} for a group $G$ acting on a set $\Omega$ is an orbit of $G$ on 
$\Omega\times \Omega$. A graph with vertex set $\Omega$ is $G$-invariant if and only if its edge-set is a union of orbitals.
Every such graph
for $G=T\times T$ acting in its diagonal action on $T$ is a Cayley graph for $T$ with connection set being
a union of conjugacy classes of $T$.

Now if $T$ has an exact factorisation $T=AB$, consider $T$ acting on the set $\Sigma$ of right cosets of $A$. Then $B$ is a regular subgroup in this action and so is a transversal for the partition $\mathcal{P}$ of $\Omega=T$ into right cosets of $A$. Following \cite[Example 3.3]{Neumann}, for all $g=(x,y)\in G$, the image of $\mathcal{P}$ under $g$ is the partition of $T$ into right cosets of $A^{x}$. Since $B$ is also a regular subgroup for the action of $T$ on the set of right cosets of $A^{x}$ it follows that $B$ is also a transversal for the partition $\mathcal{P}^g$ and so $G$ is non-synchronising. Similarly, if $A\leqslant T$ and $B$ is a sharply-transitive set for the action of $T$ on the set of right cosets of $A$ then $B$ is a transversal for the partition of $T$ into the set of right cosets of $A^x$ in $T$, for any $x\in T$. (Recall that $B\subseteq T$ is a \emph{sharply-transitive set} for the action of $T$ on a set $\Sigma$ if for all $\alpha,\beta\in \Sigma$ there is a unique $b\in B$ such that $\alpha^b=\beta$.) Hence, $G$ is non-synchronising in this case as well, see \cite[\S 6.4]{survey}.

We show that if $T=\PSL(2,9)$ then the diagonal action of $T\times T$ on $T$ is non-synchronising even though $\PSL(2,9)$ does not have an exact factorisation. 

\begin{proposition}\label{A6}
Let $T=\PSL(2,9)$ and let $G$ be $T\times T$ acting in  diagonal action on $T$. Then $G$ is non-synchronising.
\end{proposition}

\begin{proof}
First recall that $T\cong A_6$.
Let $A$ be $A_5\leqslant T$ and let $B$ be a sharply-transitive set of permutations in $A_6$,
in the action of $A_6$ on six points. (Note: $B$ is well-known to exist. For example, take 
\small{$\{ (), (1\,2)(3\,4\,5\,6), (1\,3)(2\,4\,6\,5), (1\,4)(2\,5\,3\,6), (1\,5)(2\,6\,4\,3), (1\,6)(2\,3\,5\,4)  \}$}). 
Then as discussed in the previous paragraph, $G$ is non-synchronising.
\end{proof}

\section{Analysis}

Let $T=\PSL(2,13)$ and let $G$ be $T\times T$ acting in diagonal action upon $T$. Now 
$T$ has eight conjugacy classes of nontrivial elements, which we label
\cc{2}, \cc{3}, \cc{6}, \cc{7A}, \cc{7B}, \cc{7C}, \cc{13A}, and \cc{13B}, according to the orders of the elements. 
For $I \subseteq \{\cc{2}, \cc{3}, \cc{6}, \cc{7A}, \cc{7B}, \cc{7C}, \cc{13A}, \cc{13B}\}$, we shall denote by $\Gamma_I$ the Cayley graph on $T$ whose connection set is the set of elements in the union of conjugacy classes determined by $I$. We suppress parentheses in the subscripts, so for example, $\Gamma_{\cc{3},\cc{6}}$ has as connection set the elements of order $3$ or $6$.  If $G$ were non-separating, then this would be witnessed by the existence of some nontrivial $G$-invariant graph $\Gamma$ such that $\alpha(\Gamma)\cdot \omega(\Gamma) = 1092$. We aim to show that no such graph exists, which ostensibly requires examining the 256 graphs of the form $\Gamma_I$. Without loss of generality we may assume that $I$ is non-empty, and as $\alpha(\Gamma) \cdot \omega(\Gamma)$ is invariant under graph complementation, we only need to check one of each pair of complementary graphs, thereby leaving $127$ distinct graphs to check.

Denote the union of conjugacy classes with elements of order $7$ and $13$ by \cc{7} and \cc{13}, respectively. The following proposition drastically reduces the number of graphs which must be considered down to $15$.

\begin{proposition}\label{prop:FewerGraphs}
The group $\PSL(2,13)\times \PSL(2,13)$ acting in its diagonal action on $\PSL(2,13)$ is separating if and only if $\alpha(\Gamma_I) \cdot \omega(\Gamma_I) \neq 1092$ for all $I \subseteq \{\cc{2}, \cc{3}, \cc{6}, \cc{7}, \cc{13}\}$.
\end{proposition}

To prove Proposition \ref{prop:FewerGraphs}, we shall consider two association schemes. 
Let $A_I$ be the adjacency matrix of $\Gamma_I$ with respect to $\Omega$. Up to a reordering of relations, the orbitals of $G$ describe the following association scheme
\[
\mathcal{A}=(\Omega, \{A_\cc{1}, A_\cc{6}, A_\cc{2}, A_\cc{3}, A_{\cc{7A}}, A_{\cc{7B}}, A_{\cc{7C}}, A_{\cc{13A}}, A_{\cc{13B}} \}),
\]
where $A_\cc{1}$ is the identity matrix. We shall fix this ordering of the relations.  Note the unusual ordering of the $A_i$, which corresponds to the matrix of dual eigenvalues used in the following calculations.
It is well known that $\mathcal{A}$ is equal to an association scheme\footnote{In general, the orbitals of a transitive permutation group do not form a (symmetric) association scheme. They do, however, always form a homogeneous coherent configuration. In the cases considered in this paper, the conjugacy classes are all inverse-closed and so the group scheme is symmetric.} called the \emph{group scheme} for $T$, and the algebra generated by the adjacency matrices
is the centre of the group algebra $\mathbb{R}T$.
The matrix of dual eigenvalues $Q^\mathcal{A}$ is readily computable from the character table of $T$:
\[
Q^\mathcal{A}:=\small{
\begin{bmatrix}
 1& 49& 49& 144& 144& 144& 169& 196& 196\\
 1& -7& -7& 0& 0& 0& 13& -14& 14 \\
 1& -7& -7& 0& 0& 0& 13& 28& -28 \\
 1& 7& 7& 0& 0& 0& 13& -14& -14 \\
 1& 0& 0& -24\cos{\frac{2\pi}{7}}& -24\cos{\frac{4\pi}{7}}& -24\cos{\frac{6\pi}{7}}& -13& 0& 0\\
 1& 0& 0& -24\cos{\frac{6\pi}{7}}& -24\cos{\frac{2\pi}{7}}& -24\cos{\frac{4\pi}{7}}& -13& 0& 0 \\
 1& 0& 0& -24\cos{\frac{4\pi}{7}}& -24\cos{\frac{6\pi}{7}}& -24\cos{\frac{2\pi}{7}}& -13& 0& 0 \\
 1& \frac{7}{2}(1-\sqrt{13})& \frac{7}{2}(1+\sqrt{13})& -12& -12& -12& 0& 14& 14 \\
 1& \frac{7}{2}(1+\sqrt{13})& \frac{7}{2}(1-\sqrt{13})& -12& -12& -12& 0& 14& 14 
\end{bmatrix}.}
\]

Furthermore, we may fuse the relations \cc{7A}, \cc{7B}, and \cc{7C}, and also \cc{13A} and \cc{13B} to form a second association scheme
\[
\mathcal{B}=(\Omega, \{A_\cc{1}, A_\cc{6}, A_\cc{2}, A_\cc{3}, A_\cc{7}, A_{\cc{13}} \}),
\]
with matrix of dual eigenvalues
\[
Q^\mathcal{B}:=
\begin{bmatrix}
  1&   98&  432&  169&  196&  196\\
  1&  -14&    0&   13&  -14&   14 \\
  1&  -14&    0&   13&   28&  -28 \\
  1&   14&    0&   13&  -14&  -14 \\
  1&    0&   12&  -13&    0&    0 \\
  1&    7&  -36&    0&   14&   14 
\end{bmatrix}.
\]

For both $Q^\mathcal{A}$ and $Q^\mathcal{B}$, we index the rows by the corresponding relation (e.g., \cc{7A} indexes the 5th row of $Q^\mathcal{A}$) and the columns by the corresponding projection matrix.
Now, we have a decomposition of $\mathbb{R}^\Omega$ into simultaneous eigenspaces of $\mathcal{A}$
with projection matrices $E_0, E_1, \ldots, E_8$, given by
\[
E_i = \sum_{j \in \{\cc{1}, \cc{6}, \cc{2}, \cc{3}, \cc{7A}, \cc{7B}, \cc{7C}, \cc{13A}, \cc{13B}\} } Q^{\mathcal{A}}_{ji} A_j.
\]

Let $I \subseteq \{\cc{6}, \cc{2}, \cc{3}, \cc{7A}, \cc{7B}, \cc{7C}, \cc{13A}, \cc{13B}\}$, and let $C$ and $S$ be a maximum clique and coclique of $\Gamma_I$. Now by Theorem \ref{thm:CliquesInSchemes}, $|C| \cdot |S| = 1092$ if and only if their characteristic vectors $\chi_C$ and $\chi_S$ are design-orthogonal. Let $X$ and $Y$ be the dual degree sets of $\chi_C$ and $\chi_S$ respectively, then 
\[
\chi_C = \sum_{i \in X} \chi_C E_i = \chi_C \sum_{i \in X} E_i \quad \text{ and } \quad \chi_S = \sum_{i \in Y} \chi_S E_i = \chi_S\sum_{i \in Y} E_i.
\]
Take $i \in \{1,2,3,4,5\}$ and let $v$ be the first row of $E_i$. Then $v \in \rm{Im}(E_i)$ and has both rational and irrational entries, since $v$ is a linear combination of the first rows of the $A_j$ matrices, with coefficients given by the $i$-th column of $Q^{\mathcal{A}}$. Clearly $\langle v^G \rangle$ is therefore not a $G$-submodule over $\mathbb{Q}$. Assume for a contradiction that there is a vector $u \in \rm{Im}(E_i)$ with rational entries. Then $\langle u^G \rangle$ is clearly a $G$-submodule over $\mathbb{Q}$. However, since $\rm{Im}(E_i)$ is irreducible as a $G$-submodule over $\mathbb{R}$, we require $\langle v^G \rangle = \langle u^G \rangle$, a contradiction.
 Hence there are no rational vectors in the image of $E_i$ for $i \in \{1,2,3,4,5\}$. 

It is clear from studying the irrational entries of $Q^\mathcal{A}$ that
\begin{center}
\begin{tabular}{ lll } 
$F_0 := E_0, \quad\quad$ & $F_1 := E_1 + E_2, \quad\quad$ & $F_2 := E_3 + E_4 + E_5,$\\
$F_3 := E_6,$ & $F_4 := E_7,$ & $F_5 := E_8,$
\end{tabular}
\end{center}
are the unique smallest linear combinations of the projection matrices with rational entries.
Since $\chi_C$ and $\chi_S$ have rational entries, it is true that
\[
\chi_C = \sum_{i \in \overline{X}} \chi_C F_i  \quad \text{ and } \quad \chi_S = \sum_{i \in \overline{Y}} \chi_S F_i,
\]
where $\overline{X}$ and $\overline{Y}$ are the dual degree sets of $\chi_C$ and $\chi_S$ with respect to the new $F_i$ projection matrices. Clearly $\chi_C$ and $\chi_S$ are therefore design-orthogonal with respect to the real $E_i$ projection matrices if and only if they are design-orthogonal with respect to the rational $F_i$ projection matrices.
 Thus we have reduced the number of nontrivial projection matrices that we need to study from $8$ to $5$. 

We observe that the $F_i$ projection matrices are precisely the projection matrices onto the simultaneous eigenspaces of $\mathcal{B}$, given by
\[
F_i = \sum_{j \in \{1, 6, 2, 3, 7, 13\} } Q^{\mathcal{B}}_{ji} A_j.
\]
As a result, design-orthogonality of subsets of $\Omega$ with respect to $\mathcal{A}$ is equivalent to design-orthogonality with respect to $\mathcal{B}$.
\begin{proof}[Proof of Proposition \ref{prop:FewerGraphs}]
Let $C$ be a clique and $S$ be a coclique of $\Gamma_I$ for $I \subseteq \{\cc{6}, \cc{2}, \cc{3}, \cc{7A}, \cc{7B}, \cc{7C}, \cc{13A}, \cc{13B}\}$ such that $|C|\cdot |S|=1092$. Then by Theorem \ref{thm:CliquesInSchemes}, $S$ and $C$ are design-orthogonal with respect to the association scheme $\mathcal{A}$. However, we have shown that $\mathcal{A}$ and $\mathcal{B}$ have the same projection matrices over the rationals, and so $C$ and $S$ must also be design-orthogonal with respect to $\mathcal{B}$. Applying Theorem \ref{thm:CliquesInSchemes} again, we discover that $C$ and $S$ must be a maximum clique and coclique, respectively, for some $\Gamma_{I'}$ for $I' \subseteq \{\cc{6}, \cc{2}, \cc{3}, \cc{7}, \cc{13}\}$.
\end{proof}

Proposition \ref{prop:FewerGraphs} means that we need only take unions of some $G$-invariant graphs (and thus consider fewer $G$-invariant graphs in total) when determining whether or not $G$ is non-separating.
This reduction in $G$-invariant graphs corresponds to an association scheme, $\mathcal{B}$, which is a fusion of the association scheme determined by all the $G$-invariant graphs, $\mathcal{A}$.
Moreover, we have also demonstrated that we can apply the MacWilliams transform with respect to $\mathcal{B}$ (rather than $\mathcal{A}$) to determine design-orthogonality, and hence non-separation of a $G$-invariant graph, $\Gamma_I$.
 We do so now to eliminate many of the remaining $15$ (complementary pairs of) graphs of Proposition \ref{prop:FewerGraphs}.

\begin{lemma}\label{sixgraphs}
Let $\Gamma$ be a union of nontrivial graphs of the association scheme $\mathcal{B}$ such that 
\[
\alpha(\Gamma) \cdot \omega(\Gamma) = 1092.
\]
Then the inner distribution vectors, $a$ and $b$, of a maximum clique and maximum coclique, respectively, are one of the following pairs:
\begin{enumerate}[1.]
\item $(1,0,0,0,0,12)$ and $(1,14,7,14,48,0)$.
\item $(1,0,0,0,13,0)$ and $(1,26,13,26,0,12)$.
\item $(1,0,0,26,0,12)$ and $(1,7-t,t,0,20,0)$, $0\leqslant t\leqslant 7$.
\item $(1,0,0,14, 27,0)$ and $(1,13-t,t,0,0,12)$, $0\leqslant t\leqslant 13$.
\item $(1,0,13,0,0,12)$ and $(1,14-t,0,t,27,0)$, $0\leqslant t\leqslant 14$. 
\item $(1,0,7,0,20,0)$ and $(1,26-t,0,t,0,12)$, $0\leqslant t\leqslant 26$.
\end{enumerate}
\end{lemma}

\begin{proof} Let $Q = Q^\mathcal{B}$ henceforth.
Since we are only interested in the graphs of $\mathcal{B}$ up to complementation, 
we may suppose without loss of generality that $a$ has at least three zero entries. 
The first entry of the inner distribution vector counts the average number of elements related to one another by the identity relation, and hence is always $1$, so we let
 $a:=(1,a_1,a_2,a_3,a_4,a_5)$ and let $b:=(1,b_1,b_2,b_3,b_4,b_5)$.
Note that the sum of the entries of $a$ and $b$, respectively, yields the size of the associated clique
or coclique.
Since $a$ and $b$ represent a clique and coclique for $\Gamma$, we have 
\begin{equation} \label{eq:CliqueCoclique}
a\circ b =(1,0,\ldots,0).
\end{equation}
By Theorem \ref{thm:CliquesInSchemes},
\begin{equation}
    (aQ)\circ (bQ)=(1092,0,\ldots,0). \label{eq:designortho}
\end{equation}
Moreover,

\noindent\begin{minipage}{.5\linewidth}
\footnotesize
\begin{align}
	aQ =& \nonumber \\
	\big(&\quad a_5+a_4+a_3+a_2+a_1+1, \label{eq:aQ0}\\
	&\quad 7(a_5+2a_3-2a_2-2a_1+14), \label{eq:aQ1}\\
	&\quad 12(-3a_5+a_4+36), \label{eq:aQ2}\\
	&\quad 13(-a_4+a_3+a_2+a_1+13), \label{eq:aQ3}\\
	&\quad 14(a_5-a_3+2a_2-a_1+14), \label{eq:aQ4}\\
	&\quad 14(a_5-a_3-2a_2+a_1+14) \label{eq:aQ5} \quad	\big),
\end{align}
\end{minipage}%
\begin{minipage}{.5\linewidth}
\footnotesize
\begin{align}
	bQ =& \nonumber\\
	\big(& \quad b_5+b_4+b_3+b_2+b_1+1, \label{eq:bQ0}\\
	&\quad 7(b_5+2b_3-2b_2-2b_1+14), \label{eq:bQ1}\\
	&\quad 12(-3b_5+b_4+36), \label{eq:bQ2}\\
	&\quad 13(-b_4+b_3+b_2+b_1+13), \label{eq:bQ3}\\
	&\quad 14(b_5-b_3+2b_2-b_1+14), \label{eq:bQ4}\\
	&\quad 14(b_5-b_3-2b_2+b_1+14) \label{eq:bQ5} \quad	\big).
\end{align}
\end{minipage}
\vspace{0.5cm}

Each entry of $aQ$ and $bQ$ is non-negative (because the $F_j$ are positive semidefinite), and the entries of $a$ and $b$ are non-negative rational numbers. This information yields a series of equations and inequalities. We will
undertake a step-by-step analysis. 
\begin{description}
\item[Case $a_1 = a_2 = a_3 = a_4 = 0$ and $a_5>0$]
Then (\ref{eq:aQ1}), (\ref{eq:aQ3}), (\ref{eq:aQ4}), and (\ref{eq:aQ5}) are all greater than zero.
For $a$ to be nontrivial (\ref{eq:aQ2}) must then be zero, hence $a = (1, 0, 0, 0, 0, 12)$.
By (\ref{eq:designortho}) we have (\ref{eq:bQ1}), (\ref{eq:bQ3}), (\ref{eq:bQ4}), and (\ref{eq:bQ5}) must all be zero, and (\ref{eq:bQ0}) equals $84$, and by (\ref{eq:CliqueCoclique}) we have $b_5=0$. So $b = (1, 14, 7, 14, 48, 0)$.

\item[Case $a_1 = a_2 = a_3 = a_5 = 0$ and $a_4>0$]
Then (\ref{eq:aQ1}), (\ref{eq:aQ2}), (\ref{eq:aQ4}), and (\ref{eq:aQ5}) are all greater than zero.
For $a$ to be nontrivial (\ref{eq:aQ3}) must then be zero, hence $a = (1, 0, 0, 0, 13, 0)$.
By (\ref{eq:designortho}), (\ref{eq:bQ0}) must equal $78$, and (\ref{eq:bQ1}), (\ref{eq:bQ2}), (\ref{eq:bQ4}), and (\ref{eq:bQ5}) must all be zero. Thus $b=(1,26,13,26,0,12)$.

\item[Case $a_1 = a_2 = a_4 = a_5 = 0$ and $a_3 > 0$]

Then (\ref{eq:aQ1}), (\ref{eq:aQ2}), (\ref{eq:aQ3}) are all greater than zero. Under the current assumption, for $a$ to be nontrivial, (\ref{eq:aQ4}) and (\ref{eq:aQ5}) must thus both equal to zero. Hence $a = (1, 0, 0, 14, 0, 0)$. However this means the size of the $a$ is $15$, which does not divide 1092, a contradiction. This case does not occur.

\item[Case $a_1 = a_3 = a_4 = a_5 = 0$ and $a_2> 0$]
Then (\ref{eq:aQ2}),  (\ref{eq:aQ3}),  (\ref{eq:aQ4}) are greater than zero, hence  (\ref{eq:aQ1}) and  (\ref{eq:aQ5}) must simultaneously be zero. So $a = (1, 0, 7, 0, 0, 0)$. However the size of the $a$, $8$, does not divide $1092$. This case does not occur.

\item[Case $a_2 = a_3 = a_4 = a_5 = 0$ and $a_1> 0$]
Then either $a_1 = 7$ or $a_1=14$ such that (\ref{eq:aQ1}) or (\ref{eq:aQ4}) is zero. In either case the size ($8$ or $15$) does not divide $1092$. This case does not occur.

\item[Case $a_1 = a_2 = a_3 =0$ and $a_4,a_5> 0$] 
Now, \eqref{eq:aQ1}, \eqref{eq:aQ4}, and \eqref{eq:aQ5} are all non-zero, hence \eqref{eq:bQ2}, \eqref{eq:aQ4}, \eqref{eq:aQ5} are all zero according to \eqref{eq:designortho}. We have $b_4=b_5=0$ by (\ref{eq:CliqueCoclique}), meaning $b = (1, 14, 7, 14, 0, 0)$. However the size of $b$, $36$, does not divide $1092$. This case does not occur.

\item[Case $a_1=a_2=a_4=0$ and $a_3,a_5> 0$]

Then $b_3=b_5=0$ and we have $a_3=26$, $a_5=12$, $b_4=20$ as well. Moreover, $a=(1,0,0,26,0,12)$ and $b=(1,7-b_2,b_2,0,20,0)$, $0\le b_2\le 7$.

\item[Case $a_1=a_2=a_5=0$ and $a_3,a_4> 0$]

Then $b_3=b_4=0$ and we have $b_5=12$, $a_3=14$ as well.
Moreover, $a=(1,0,0,14, 27,0)$ and $b=(1,13-b_2,b_2,0,0,12)$, $0\le b_2\le 13$.

\item[Case $a_1=a_3=a_4=0$ and $a_2,a_5> 0$]

If $a_2,a_5\ne 0$, then $b_2=b_5=0$ and we have $a_2=13$, $a_5=12$, $b_4=27$. Moreover, $a=(1,0,13,0,0,12)$ and $b=(1,14-b_3,0,b_3,27,0)$, $0\le b_3\le 14$.

\item[Case $a_1=a_3=a_5=0$ and $a_2,a_4> 0$]

If $a_2,a_4\ne 0$, then $b_2=b_4=0$ and we have $a_2=7$, $a_4=20$, and $b_5=12$. Moreover, $a=(1,0,7,0,20,0)$ and $b=(1,26-b_3,0,b_3,0,12)$, $0\le b_3\le 26$.

\item[Case $a_1=a_4=a_5=0$ and $a_2,a_3> 0$]

So $b_2=b_3=0$. Straight away, we have \eqref{eq:aQ2}, \eqref{eq:aQ3}, \eqref{eq:bQ5}$>0$
and so $-3b_5+b_4+36=0$, $-b_4+b_1+13=0$,
and $-a_3-2a_2+14=0$. Hence \eqref{eq:aQ4} $>0$
and $b_5-b_1+14=0$. However, then \eqref{eq:bQ1} $<0$,
which is a contradiction.
This case does not occur.

\item[Case $a_2=a_3=a_4=0$ and $a_1,a_5> 0$]

Here we have $b_1=b_5=0$, $a_5=12$ and two cases:
\begin{itemize}
    \item $a=(1,13,0,0,0,12)$, $b=(1,0,0,14,27,0)$.
    \item $a=(1,26,0,0,0,12)$, $b=(1,0,7,0,20,0)$.
\end{itemize} 
These cases already appear above.

\item[Case $a_2=a_3=a_5=0$ and $a_1,a_4> 0$]

Here, we have $b_1=b_4=0$, $b_5=12$ and two cases:
\begin{itemize}
\item $a=(1,7,0,0,20,0)$ and $b=(1,0,0,26,0,12)$.
\item $a=(1,14,0,0,27,0)$ and $b=(1,0,13,0,0,0)$.
\end{itemize}
These cases already appear above.

\item[Case $a_2=a_4=a_5=0$ and $a_1,a_3> 0$]

If $a_1,a_3\ne 0$, then 
\[
b=(1,0,91/5,0,156/5,112/5),
\]
which yields a size that is not an integer, so this case does not occur.

\item[Case $a_3=a_4=a_5=0$ and $a_1,a_2> 0$]

If $a_1,a_2\ne 0$, then \[
b=(1,0,0,91/2,117/2,63/2),
\]
which yields a size that is not an integer. This case does not occur.\qedhere
\end{description}
\end{proof}

Each combination of the inner distribution vectors given by Lemma \ref{sixgraphs} defines a complementary pair of graphs
of $\mathcal{B}$. For example, consider the fourth case where the inner distribution vectors are
$(1,0,0,14, 27,0)$ and $(1,13-t,t,0,0,12)$. Upon inspection of the positions of zeroes of these vectors we have three cases to consider. In the case that $0 < t < 13$, we are interested in $\omega(\Gamma_{\cc{3,7}})$ and $\omega(\Gamma_{\cc{6,2,13}})$. In the case that $t=0$, we are interested in $\omega(\Gamma_{\cc{3,7}})$ and $\omega(\Gamma_{\cc{6,13}})$. In the case that $t=13$, we are interested in $\omega(\Gamma_{\cc{3,7}})$ and $\omega(\Gamma_{\cc{2,13}})$. Since $\Gamma_{\cc{6,13}}$ and $\Gamma_{\cc{2,13}}$ are spanning subgraphs of $\Gamma_{\cc{6,2,13}}$, and  $\Gamma_{\cc{6,2,13}}$ is the complementary graph of $\Gamma_{\cc{3,7}}$, it is sufficient for us to consider $\omega(\Gamma_{\cc{3,7}})$ and $\alpha(\Gamma_{\cc{3,7}})$ in all cases. Thus the complementary pair of graphs that we are taking a maximum clique and coclique from is $\Gamma_{\cc{3,7}}$ or $\Gamma_{\cc{6,2,13}}$. 

We now have just six graphs to analyse, and by Lemma \ref{sixgraphs}, we know the only possibilities for $\alpha(\Gamma)$ and $\omega(\Gamma)$ when $\alpha(\Gamma) \cdot \omega(\Gamma) = 1092$.
We record this information in Table \ref{tab:cliquebounds}. Each of the six graphs has either a clique or a coclique of the size determined by Lemma~\ref{sixgraphs}, these are indicated by an asterisk in the table.

\begin{table}[!ht]
\begin{center}
\begin{tabular}{clcc} 
\toprule
Case&Graph $\Gamma$ & $\omega(\Gamma)$ & $\alpha(\Gamma)$ \\
\midrule
1&$\Gamma_{\cc{13}}$ & 13* & 84\\ 
2&$\Gamma_{\cc{7}}$ & 14 & 78* \\ 
3&$\Gamma_{\cc{3,13}}$ & 39* & 28\\ 
4&$\Gamma_{\cc{3,7}}$ & 42 & 26* \\ 
5&$\Gamma_{\cc{6,13}}$ & 26* & 42\\ 
6&$\Gamma_{\cc{6,7}}$ & 28 & 39*\\ 
\bottomrule
\end{tabular}
\end{center}
\caption{Putative non-separating graphs to be considered for $q=13$.}\label{tab:cliquebounds}
\end{table}

We write $\alpha_I$ and $\omega_{I}$ for the coclique and clique numbers of $\Gamma_I$.
We shall say that $\Gamma_I$ is \emph{separating} if $\alpha_I \cdot \omega_I \neq 1092$ (so $G$ is separating if and only if all the graphs $\Gamma_I$ are separating).
Recall that if a graph $\Delta$ is a spanning subgraph of a graph $\Gamma$,
then $\alpha(\Delta)\ge \alpha(\Gamma)$ and $\omega(\Delta)\le \omega(\Gamma)$.

\begin{lemma}\label{twographsdone}\leavevmode
\begin{enumerate}[(i)]
\item If $\alpha_{\cc{6},\cc{13}}<42$, then $\omega_{\cc{3},\cc{7}}<42$ and the graph $\Gamma_{\cc{3},\cc{7}}$ is separating.
\item If $\alpha_{\cc{3},\cc{13}}<28$, then $\omega_{\cc{6},\cc{7}}<28$ and the graph $\Gamma_{\cc{6},\cc{7}}$ is separating.
\end{enumerate}
\end{lemma}

\begin{proof}
Since $\Gamma_{\cc{3},\cc{7}}$ is an spanning subgraph of $\Gamma_{\cc{2},\cc{3},\cc{7}}$, we have
$\omega_{\cc{3},\cc{7}}\le \omega_{\cc{2},\cc{3},\cc{7}}=\alpha_{\cc{6},\cc{13}}$.
Likewise, $\omega_{\cc{6},\cc{7}}\le \omega_{\cc{2},\cc{6},\cc{7}}=\alpha_{\cc{3},\cc{13}}$.
\end{proof}

\begin{lemma}\label{gamma7}
$\Gamma_\cc{7}$ is separating.
\end{lemma}

\begin{proof}
Consider the natural action of $T$ on a set $\Sigma$ of size 14. Then a point stabiliser $T_\omega$ forms a
coclique of $\Gamma_\cc{7}$. To see why, notice that if $x,y\in T_\omega$, then
$xy^{-1}\in T_\omega$, and in particular, $xy^{-1}$ does not have order 7 (because $|T_\omega|=6\times 13$).
Moreover, by Table \ref{tab:cliquebounds}, $T_\omega$ yields a coclique of maximum size. Suppose that $C$ is a clique of $\Gamma_{\cc{7}}$. If $g,h\in C$ such that there exists $\omega\in\Sigma$ for which $\omega^g=\omega^h$, then $gh^{-1}\in T_\omega$. However, this contradicts the elements of $T_\omega$ having order coprime to 7.   Thus $C$ is a sharply-transitive set of permutations in $T$
(on $\Sigma$). Such a set does not exist (cf., \cite{mathoverflow}).
\end{proof}

So we have reduced the problem to three graphs: $\Gamma_{\cc{13}}$, $\Gamma_{\cc{6},\cc{13}}$, $\Gamma_{\cc{3},\cc{13}}$.

\begin{lemma} \label{TwoOfTheTrickyOnes}
$\alpha(\Gamma_{\cc{6},\cc{13}})=25$ and $\alpha(\Gamma_{\cc{3},\cc{13}})=22$. Therefore, 
$\Gamma_{\cc{6},\cc{13}}$ and $\Gamma_{\cc{3},\cc{13}}$ are separating.
\end{lemma}

\begin{proof}
We used \textsc{Grape} \cite{grape}, and its basic clique-finding algorithm on a personal laptop. It took
less than 10 minutes each in each case to determine that $\alpha(\Gamma_{\cc{6},\cc{13}})=25$
and $\alpha(\Gamma_{\cc{3},\cc{13}})=22$.
\end{proof}

\begin{lemma} \label{gamma13}
$\Gamma_{\cc{13}}$ is separating.
\end{lemma}

\begin{proof}
From Table \ref{tab:cliquebounds}, a maximum clique in $\Gamma_{13}$ has size at most 13. Let $C$ be a cyclic subgroup of $T$ of order 13. 
Then the elements of $C$ form a maximum clique of $\Gamma_{\cc{13}}$. To see why,
let $g,h\in C$. Then $gh^{-1}\in C$ and hence $gh^{-1}$ has order 13.
 Note that $\Aut(\Gamma_{\cc{13}})$
acts transitively on the vertices of $\Gamma_{\cc{13}}$ as it contains $G$.
Let $M$ be the matrix whose rows are the characteristic vectors of each $C^g$,
where $g\in \Aut(\Gamma_{\cc{13}})$. Then the characteristic vector of a coclique is a $\{0,1\}$--vector
$v$ such that $Mv\preccurlyeq (1,1,\ldots, 1)$. This gives us an integer linear program,
where our objective function is the sum of the values of $v$.
Using the mixed integer linear programming (MILP) software \textsf{Gurobi} \cite{gurobi}, we find 
that $\alpha(\Gamma_{\cc{13}})$ is 78. However, $13\times 78<1092$
and hence $\Gamma_{\cc{13}}$ is separating.
\end{proof}

\begin{proof}[Proof of Theorem \ref{main1}]
For $q=13$, the proof follows from 
Proposition \ref{prop:FewerGraphs}, and Lemmas \ref{sixgraphs}, \ref{twographsdone}, \ref{gamma7}, \ref{TwoOfTheTrickyOnes}, and \ref{gamma13}.
For $q=17$ the analysis proceeds in the same manner. An analogous result to Proposition \ref{prop:FewerGraphs} reduces the number of $G$-invariant graphs which need to be considered from $255$ to $31$. With the aid of the computer algebra package, \textsf{Mathematica} \cite{mathematica}, a result analogous to Lemma \ref{sixgraphs} was obtained, leaving only $23$ graphs for consideration. These graphs were ruled out computationally by using \textsc{Grape} \cite{grape} in some instances, and in other instances, we can formulate a constraint satisfaction problem from the fact (see Corollary \ref{intersect1}) that a maximum clique and maximum coclique would have to meet in precisely one element (as described in the proof of Lemma \ref{gamma13}). We used \textsf{Gurobi} \cite{gurobi} for these constraint satisfaction problems, and verified some cases with \textsf{Minion} \cite{minion}.
See Table \ref{tbl:graphs17} in Appendix \ref{append:Cases17} for details on these graphs, including which bounds were obtained (indicated by an asterisk), and which methods were used. Graphs $\Gamma_I$ are labelled according to the conjugacy classes of elements of order $2, 3, 4, 8, 9, 17$ in the same manner as for $q=13$. For the reader interested in verifying these computational results, supplementary materials have been made available at \cite{SupplementaryMaterials}.
\end{proof}

\section{Non-spreading}

We already know from our knowledge of exact factorisations of $\PSL(2,q)$ that $G:=\PSL(2,q)\times \PSL(2,q)$ in its diagonal action
is non-separating, and hence non-spreading,
when $q$ is even or when $q\equiv 3\pmod{4}$. So let $q$ be an odd prime power such that $q\equiv 1\pmod{4}$, and let $T=\PSL(2,q)$.
Let $T_1$ be a point stabiliser in the natural action of $T$ on $q+1$ points.
Next, let 
\[
T_1^2:=\{t^2 : t\in T\}.
\]
This setup will provide us with the ingredients we need to exhibit a multiset $A$
and a set $B$ with the desired properties of a witness to $G$ being non-spreading.

\begin{lemma}\label{nicepropPSL}\samepage
Let $T=\PSL(2,q)$, with $q\equiv 1\pmod{4}$, and let $T_1$ and $T_2$ be two different point stabilisers in the natural
$2$-transitive action of $T$ (on $q+1$ points). Then 
\[
|T_1^2\cap T_2t|=\tfrac{1}{2}|T_1\cap T_2t|
\]
for all $t\in T$.
\end{lemma}

\begin{proof}
Since $T$ is $2$-transitive, we may suppose that $T_1$ and $T_2$ are the stabilisers of the points $1$ and 
$2$ in the action of $T$ on $\{1,\ldots, q+1\}$.
First, we note that $T_1$ has structure $[q] \sd C_{(q-1)/2}$ and $T_1^2$ is actually a subgroup with structure
$[q]\sd C_{(q-1)/4}$. Let $t\in T$. Now $T_1\cap T_2t$ is the set of elements of $T$ that fix $1$ and map $2$ to $2^t$.
The stabiliser $T_1\cap T_2$ is the top group $C_{(q-1)/2}$ acting semiregularly on the remaining points,
with two orbits $\mathcal{O}$ and $\mathcal{O}'$, each of size $(q-1)/2$. Therefore,
$|T_1\cap T_2t|$ is equal to $0$ or $(q-1)/2$, depending on whether $2^t=1$ or not.
Now consider $T_1^2\cap T_2t$. Here, the top group of $T_1^2\cap T_2$ is $C_{(q-1)/4}$ 
and it splits each orbit $\mathcal{O}$ and $\mathcal{O}'$ in half. 
Therefore, $|T_1^2\cap T_2t|$ is equal to $0$ or $(q-1)/4$, depending on whether $2^t=1$ or not.
\end{proof}

\begin{proof}[Proof of Theorem \ref{main2}]
Let $A$ be the multiset defined in the following way:
place a multiplicity of 2 on each element of $T_1^2$; place a multiplicity of 1 on each element of $T\backslash T_1$.
We will show that 
\begin{enumerate}[(i)]
\item $|A|=|T|$;
\item $|\chi_A\circ \chi_{T_1^g}|=|T_1|$ for all $g\in G$.
\end{enumerate}

First, $|A|=2|T_1^2|+|T|-|T_1|=|T|$, so (i) is satisfied. Next, suppose $g\in G$.
Then
\begin{align*}
|\chi_A\circ \chi_{T_1^g} |&=|T_1^2\cap T_1^g|-|(T_1\backslash T_1^2)\cap T_1^g|+|T_1^g|\\
&=|T_1^2\cap T_1^g|-|T_1\cap T_1^g|+|T_1^2\cap T_1^g|+|T_1|\\
&=2|T_1^2\cap T_1^g|-|T_1\cap T_1^g|+|T_1|.
\end{align*}
Now by Lemma \ref{nicepropPSL},
$|T_1^2\cap T_1^g|=\tfrac{1}{2}|T_1\cap T_1^g|$, because $T_1^g$ is a right coset (in $T$) of a point-stabiliser in $T$
(in the natural $2$-transitive action of $T$).
Therefore, $|\chi_A\circ \chi_{T_1^g} |=|T_1|$ and so (ii) is satisfied. Therefore, the pair
$(A,T_1)$ witness that $G$ is non-spreading.
\end{proof}

\subsection*{Acknowledgements} This work forms part of an Australian Research Council Discovery Project
DP200101951.

\bibliographystyle{abbrv}
\bibliography{references}

\newpage
\appendix

\section{Cases for $q=17$.} \label{append:Cases17}

\begin{table}[!ht]
\begin{center}
\begin{tabular}{lccl}
\toprule
 $I$  & $\alpha$ to be satisfied & $\omega$ to be satisfied & Comment \\
\midrule
 \cc{2}, \cc{4}, \cc{8}, \cc{9}, \cc{17}  & 9 & 272 & \textsc{Grape} ($\alpha= 3$) \\
 \cc{2}, \cc{3}, \cc{4}, \cc{8}, \cc{9} & 17* & 144 & constraint sat. solver\\
 \cc{2}, \cc{3}, \cc{4}, \cc{8}, \cc{17} & 18 & 136 & \textsc{Grape} ($\alpha= 10$) \\
 \cc{2}, \cc{8}, \cc{9}, \cc{17} & 18 & 136 & \textsc{Grape} ($\alpha=6$) \\
 \cc{2}, \cc{3}, \cc{8}, \cc{9} & 34* & 72 & constraint sat. solver \\
 \cc{2}, \cc{3}, \cc{8}, \cc{17} & 36 & 68 & \textsc{Grape} ($\alpha=18$) \\
 \cc{2}, \cc{4}, \cc{9}, \cc{17}  & 18 & 136 & \textsc{Grape} ($\alpha=10$) \\
 \cc{4}, \cc{8}, \cc{9}, \cc{17} & 18 & 136 & \textsc{Grape} ($\alpha=12$) \\
 \cc{2}, \cc{4}, \cc{8}, \cc{17} & 18 & 136 & \textsc{Grape} ($\alpha=11$) \\
 \cc{2}, \cc{3}, \cc{4}, \cc{9} & 34* & 72 & constraint sat. solver \\
 \cc{2}, \cc{3}, \cc{4}, \cc{17} & 36 & 68* & constraint sat. solver \\
 \cc{3}, \cc{4}, \cc{8}, \cc{9} & 34* & 72 & constraint sat. solver \\
 \cc{2}, \cc{3}, \cc{4}, \cc{8} & 34 & 72 & \textsc{Grape} ($\alpha=17$) \\
 \cc{3}, \cc{4}, \cc{8}, \cc{17} & 36 & 68 & \textsc{Grape} ($\alpha=12$) \\
 \cc{2}, \cc{9}, \cc{17} & 36 & 68 & \textsc{Grape} ($\alpha=16$) \\
 \cc{8}, \cc{9}, \cc{17} & 36 & 68 & \textsc{Grape} ($\alpha=24$) \\
 \cc{2}, \cc{8}, \cc{17} & 36 & 68* & constraint sat. solver \\
 \cc{2}, \cc{3}, \cc{9} & 68* & 36 & \textsc{Grape} ($\omega= 18$) \\
 \cc{2}, \cc{3}, \cc{17} & 72 & 34* & constraint sat. solver\\
 \cc{3}, \cc{8}, \cc{9} & 68* & 36 & constraint sat. solver\\
 \cc{2}, \cc{3}, \cc{8} & 68 & 36 & \textsc{Grape} ($\omega=13$) \\
 \cc{3}, \cc{8}, \cc{17} & 72 & 34* & constraint sat. solver \\
 \cc{3}, \cc{4}, \cc{17} & 72 & 34* & constraint sat. solver \\
\bottomrule
\end{tabular}
\end{center}
\caption{Putative graphs to be considered for $q=17$.}\label{tbl:graphs17}
\end{table}

\end{document}